\documentclass[10pt,oneside]{amsart}
\usepackage{amssymb, amscd, amsmath, amsthm, latexsym, hyperref,amsrefs,enumitem}
\usepackage[text={6.0in,9.0in},centering,letterpaper,dvips]{geometry}

\usepackage{enumitem}\setlist[enumerate,1]{font=\upshape, itemsep=.5ex}\setlist[itemize,1]{font=\upshape, itemsep=.5ex}
\usepackage{pinlabel}
\usepackage{caption}
\usepackage{xcolor}

\def\Z{{\mathbb Z}}

\def\Q{{\mathbb Q}}

\def\calt{\mathcal{T}}
\def\calh{\mathcal{H}}

\def\calc{\mathcal{C}}
\def\calm{\mathcal{M}}

\def\calj{\mathcal{J}}

\def\cala{\mathcal{A}}

\def\hp{\calh_{(p)}}

\newcommand{\compactlist}{\begin{list}{\enumerate}{\setlength{\leftmargin}{1em}}}

\def\cs{\mathbin{\#}}
\def\co{\colon\thinspace}

\def\ot{\overline{\tau}}

\theoremstyle{theorem}

\newtheorem{theorem}{Theorem}
\newtheorem{lemma}[theorem]{Lemma}
\newtheorem{corollary}[theorem]{Corollary}

\theoremstyle{definition}

\AtBeginDocument{%
\def\MR#1{}
}
\begin{document}
\title[Rank-expanding satellite operators]{Rank-expanding satellite operators on the topological knot concordance group}

\author{Charles Livingston}
\address{Charles Livingston: Department of Mathematics, Indiana University, Bloomington, IN 47405}\email{livingst@indiana.edu}



\begin{abstract}  Given a  fixed knot  $P \subset S^1 \times B^2$ and any knot $K \subset S^3$, one can form the satellite of $K$ with pattern $P$.  This operation  induces a self-map of  the  concordance group of knots in $S^3$.  It has been proved by Dai, Hedden, Mallick, and Stoffregen that in the smooth category there exist    $P$  for which this function is rank-expanding; that is, for some $K$, the set $\{P(nK)\}_{n\in \Z}$ generates an infinite rank subgroup.  Here we demonstrate that similar examples exist in the case of the  topological locally flat concordance group.   Such examples cannot exist in the algebraic concordance group. \end{abstract}

\maketitle


\section{Main theorem} \label{sec:intro}

Let $P\subset S^1 \times B^2$ be a knot.  For any knot $K \subset S^3$, let $P(K)$ denote the satellite of $K$ formed using $P$ as the {\it pattern}.  For $* = s, t $ or $a$, $P$ induces a map  $P_* \co \calc_* \to \calc_*$, where $\calc_*$ is either the smooth, topological, or algebraic knot concordance group.  These maps commute with the natural surjections $\calc_s \to \calc_t \to \calc_a$.

Let $M$ be a $\Z$--module.  The {\it rank} of $M$ is the dimension of the vector space $M\otimes_\Z \Q$.
A   function $\psi\co M \to M$   is called {\it rank-expanding} if there exists an $m \in   M$ for which $\psi(\left<m \right>)$
generates an  infinite rank submodule.     This is equivalent to the condition that there exists an $m \in M$ such that $\psi(\left< m \right>)$
contains an infinite linearly independent subset.

Questions related to determining whether  maps $P_s$ as defined above are rank-expanding  appeared  in~\cite{MR4270665}. The problem came into full focus in~\cite{2022arXiv220907512D}.  Both of these papers applied only in the smooth category.  The  main results of~\cite{2022arXiv220907512D} demonstrated that for large families of $P$, the maps  $P_s$ are rank-expanding.  A deeper aspect of that work is that it   yielded results when restricted to $\calt \subset \calc_s$, the concordance group of topologically slice knots.  Other references about  concordance maps  induced by   satellite operations include~\cite{MR3286894, 2022arXiv220905400D, MR4299664, MR3493412, MR3589337, MR827271, MR534409, 2022arXiv220714198L, 2019arXiv191003461M, 2020arXiv201011277C}.

In contrast to the situation  in the smooth category, in the algebraic setting $P_a$ is never rank-expanding.  In fact,    $P_a$  is affine; if we denote by $[K]$ the algebraic concordance class of $K$, then $P_a([K]) = [P(U)] +  \phi([K])$,  where $\phi$ is a homomorphism that depends  only on the winding number.   This fact follows readily from results of Seifert~\cite{MR0035436} and the definition  of the algebraic concordance group in terms of Seifert matrices~\cite{MR246314}.    As a consequence, if  a subgroup $G\subset \calc_a$ is generated by a set of $k$ elements, then    $P_a(G)$ generates a subgroup of rank at most $k+1$.

In this paper we will work in the topological locally flat category and  use Casson-Gordon theory to  prove that $P_t$ can be rank-expanding.

To state the main result, we need to set up some notation.     For a knot $P\subset S^1 \times B^2$ we let   $M_2(P(K))$ denote the 2--fold cover of $S^3$ branched over $P(K)$.
The Levine-Tristram signature function~\cite{MR0248854,MR0253348}  of  a knot $K$ evaluated at $\omega = e^{2 k  \pi i /p}$ is denoted $\sigma_{k/p}(K)$.   In terms of a Seifert matrix $V$ for $K$, it is the signature of $(1-\omega)V + (1 - \overline{\omega})V^{\sf T}$.

View $S^1 \times B^2 $ as $S^3 \setminus \nu(V)$ for some unknot $V$ in $S^3$, where $\nu(V)$ denotes the interior of a closed tubular neighborhood of $V$.  If $P$ has even winding number then the preimage of $V$ in $M_2(P(U))$ consists of  a pair of curves which we denote $V_1(P)$ and $V_2(P)$.

\begin{theorem}[Main Theorem]
\label{thm:main}  Let $P$ be an even winding number knot in $S^1 \times B^2$ satisfying   $H_1(M_2(P(U))) \cong \Z_n$ for some $n>1$ and assume that $V_1(P)$ generates $H_1(M_2(P(U))$.    Let $ \calj  = \{J_i\}$ be an infinite set of knots satisfying $\max_{1 \le j < n/2} \{\sigma_{j/n}(J_i)\} < \min_{1 \le j < n/2} \{\sigma_{j/n}(J_{i+1})\}$  for all $i \ge 1$.  Then for some infinite   set of positive integers $\cala$,  the set of knots $\{P(J_\alpha)\}_{\alpha \in \cala}$ is linearly independent in the concordance group.
\end{theorem}

\smallskip

\noindent{\bf Example.}   Let $P(a,b)$ denote the $a$--full twisted Whitehead pattern with the clasp having $b$--full twists.  Thus, the standard  untwisted Whitehead double is $P(0,\pm 1)$.  Casson and Gordon's original work~\cite{MR900252} showed that $P(k(k+1), -1)(U)$ is algebraically slice for all $k \ge 0$, but is slice if and only if $k = 0$ or $k = 1$.  These are all {\it rational unknotting number} one as defined in~\cite{2022arXiv220907512D}.

According to~\cite{MR593626}, the 2--fold branched cover $M_2(P(a,b)(U))$ is built as surgery on the Hopf link with surgery coefficients $2a$ and $2b$.  If the meridians of the Hopf link are denoted $m_1$ and $m_2$, then the $m_i$ generate the first homology and satisfy $m_2 + 2a m_1 = 0$ and $m_1 + 2bm_2 = 0$.    If follows that the order of the first homology is $|4ab -1|$.  The first relation expresses $m_2$ in terms of  $m_1$, so we see that the first homology is cyclic, generated by $m_1$.

The curves $V_1(P)$ and $V_2(P)$ are both parallel to the  meridian $m_1$.  Thus, $P(a,b)$  satisfies the conditions of the theorem as long as $ab\ne 0$.    (The knot $P(a,0)$ is trivial.  In the  case of $P(0,b)$ the Alexander polynomial is trivial and $P(0,b)(K)$ is topologically slice for all $K$.  In~\cite{2022arXiv220907512D} it was shown that for $b \ne 0$, $P(0,b)$ is rank-expanding in the smooth setting.)

To apply Theorem~\ref{thm:main} we require a family $\calj$.  This is provided by multiples of the torus knot $T(2,n)$ for some large $n$.  The signature $\sigma_x(T(2,n))$ is positive if $x> 1/2n$.  It follows that $P(a,b)$ is rank-expanding, as demonstrated by considering the image of multiples of  $T(2, |4ab -1|)$.

\smallskip

\noindent{\bf Comment 1.}   The condition that $V_1(P)$ generates the first homology can be weakened,  but some condition is clearly necessary.  This condition rules out examples for which $P$ is contained in a three-ball, in which case the theorem wouldn't hold.

\smallskip
\noindent{\bf Comment 2.}  If $| 4ab -1|$ is a perfect square, then $P(a,b)$ is algebraically  slice and, since the winding number is 0, all satellites formed using such $P(a,b)$ as the pattern are algebraically slice.

\section{Preliminary lemma}

\begin{lemma} \label{lem:lemma1}  With $P$ as in the main theorem, for every knot $K$,  $H_1(M_2(P(K))) \cong \Z_n$.  There is a canonical isomorphism $H_1(M_2(P(U)) \cong  H_1(M_2(P(K))$.\end{lemma}
\begin{proof}

The knot $P(K)$ is formed from $P(U)$ by removing $\nu(V)$ and replacing it with a copy of the complement of $K$.  Thus, the space $M_2(P(K))$ is built from $M_2(P(U))$ by removing $\nu(V_1(P))$ and $\nu(V_2(P))$ and replacing them with copies of the complement of $K$,  $S^3 \setminus \nu(K)$.

Recall that  $H_1(S^3 \setminus \nu(K)) \cong \Z$.  A Mayer-Vietoris argument shows that  $H_1(M_2(P(U)))$ and $H_1(M_2(P(K)))$ are both formed as the quotient of   $H_1(M_2(P(U)) \setminus \{\nu(V_1(P)) \cup \nu(V_2(P))\})$  by the subgroup generated by the meridians of $V_1$ and $V_2$.  This provides the canonical isomorphism.
\end{proof}

\section{Casson-Gordon invariant results}

  For any knot $K$ and prime $p$ we let $\calh_{(p)}(K)$ denote the $p$--primary summand of $H^1(M_2(K), \Q/\Z)$.  Elements of  $\calh_{(p)}(K)$ can be thought of as homomorphisms of $H_1(M_2(K))$ to $\Q/\Z$ of order a power of $p$.  Using duality (formally, the linking form on $H_1(M_2(K))$), the group $\calh_{(p)}(K)$ can  be identified  with  $H_1(M_2(K))_{(p)}$,   the $p$--primary part of the first homology.

Let $\chi \in  \calh_{(p)}(K)$.    We set  $\ot_K(\chi) = \tau_K(\chi) - \tau_K(0) \in \Q$, where $\tau_K$ is the Casson-Gordon invariant defined in~\cite{MR900252}.  It is evident that  $\ot_K(0) = 0$.  Two basic properties are  that $\ot_K(-\chi) = \ot_K(\chi)$ and that $\ot_K$ is additive under connected sum~\cite{MR711523}.


Here is a result of Gilmer~\cite{MR711523} and  Litherland~\cite{MR780587}.  In the statement, the value of the subscript   $\chi(V_1(P))$  in $\sigma_{\chi(V_1)}$ is viewed as an  element in $\Q/\Z$ using the natural embedding $\Z_n \subset \Q/\Z$.

\begin{theorem} Let  $P\subset S^1 \times B^2$ be as above and let $\chi \in \calh_{(p)}(K)$.  For all knots $K$ \[
\ot_{P(K)}(\chi)  = \ot_{P(U)}(\chi) + 2\sigma_{\chi(V_1(P))}(K),\] where $V_1(P)$ is defined above.
\end{theorem}

The main Casson-Gordon theorem concerning slice knots implies the following when restricted to the case of interest here.

\begin{theorem}  If $K$ is slice and $\calh_{(p)}(K) \cong (\Z_{p^k})^{n}$, then $kn$ is even and  there is an order $p^{kn/2}$ subgroup   $\calm_{(p)}\subset \calh_{(p)}(K)$  satisfying $\ot_K(\chi) = 0$ for all $\chi \in \calm_{(p)}$.
\end{theorem}

Applying  this to the case of linear combinations of the knots we are considering yields the following.

\begin{corollary}\label{thm:gillith}
Let  $P\subset S^1 \times B^2$ be as above   and let $\{K_i\}$ be a sequence of knots for which $\calh_{(p)}(K) \cong \Z_{p^k}$.   If    the connected sum ${\bf K} = \cs_i^{n} P(K_i)$ is slice, then  there is an order $p^{nk/2}$ subgroup  $\calm_{(p)} \subset \calh_{(p)} ({\bf K})$ with the property that if $(\chi_1, \ldots , \chi_{n}) \in \calm_{(p)}$, then
\[
\sum_{\alpha \in \cala}( \ot_{P(U)}(\chi_\alpha) + 2\sigma_{\chi_\alpha(V_1(P)) }(K_\alpha)  )= 0,
\]
where the parameter set  $\cala$ selects the knots $P(K_i)$  for which the corresponding $\chi_i \ne 0$.
\end{corollary}

\noindent

\section{Linaer independence}  Here we prove the main theorem.

We start out by  replacing   $\calj$ with an infinite subset so that the following holds:
\[   0 < \ot(P(U), \chi) + 2 \sigma_{\chi(V_1(P))}(J_i) < \ot(P(U), \chi') + 2 \sigma_{\chi'(V_1(P))}(J_{i+1}) \]
for all $\chi \ne 0$, $\chi' \ne 0$, and   $i\ge 1$.  This is possible because there is only a finite set of values of $\ot(P(U), \chi) $ and the signature functions of the $J_i$ are assumed to be going to infinity as described in the statement of the theorem.

If the set of knots $  \{P(J_i)\}_{J_i \in \calj}$ is linearly dependent, we can find knots $K_i \in \calj$ and $L_i \in \calj$ such that following concordance relation holds.
\[  {\bf K} = \big(\cs_{i = 1}^m  P(K_i) \big) \cs  - \big(\cs_{j = 1}^n   P(L_i) \big) =  0.
\]
The $K_i$ need not be distinct  and the $L_i$ need not be distinct, but we can assume that $\{K_i\} \cap \{L_i\} $ is empty.
We can also assume that the indices for the $K_i$ and $L_i$ are such that they form nondecreasing sequences with respect to the ordering on $\calj$.

The group $\hp({\bf K})$ splits as $\hp^1 \oplus \hp^2$ where $\hp^1 = \hp (\cs_{i = 1}^m  P(K_i) )$ and  $\hp^2 = \hp (-\cs_{i = 1}^n  P(L_i)) $.  The orders of these summands are $p^{km}$ and $p^{kn}$, respectively.

Let $\calm^1_{(p)}$ and $\calm^2_{(p)}$ denote the image of the  projections of $\calm_{(p)}$ to $\hp^1$ and $\hp^2$, respectively.

\begin{lemma}\label{lem:lemma2}  For each $\chi_1 \in \calm^1_{(p)}$, there is a unique $\chi_2 \in \calm^2_{(p)}$ such that $(\chi_1, \chi_2) \in \calm_{(p)}$.
\end{lemma}
\begin{proof}  Clearly, for each $\chi_1 \in \calm^1_{(p)}$ there exists an  $\chi_2 \in \calm^2_{(p)}$ such that $(\chi_1,\chi_2) \in \calm_{(p)}$.  Suppose that $(\chi_1, \chi_2) \in \calm_{(p)}$ and $(\chi_1, \chi_2') \in \calm_{(p)}$ for some $\chi_2 \ne \chi_2'$.  Then $(0, \chi_2 - \chi_2') \in \calm_{(p)}$.  This would imply the vanishing of a Casson-Gordon invariant $\ot(\cs L_i,   \chi_2 - \chi_2')$. By additivity, this is the the sum of  nonnegative numbers, at least one of which is   positive, and so cannot be 0.
\end{proof}

\begin{lemma}\label{lem:lemma3}  There is an equality of orders $\big|  \calm_{(p)} \big| = \big|  \calm^1_{(p)} \big| = \big|  \calm^2_{(p)}\big|  =  p^{nk/2}$.
\end{lemma}
\begin{proof} Lemma~\ref{lem:lemma2} provides an injection $\calm^1_{(p)}\to \calm^2_{(p)}$.  Similarly there is an injection  $\calm^2_{(p)}\to \calm^1_{(p)}$, implying that $\calm^1_{(p)}$ and $\calm^2_{(p)}$ have the same order.  The projection map $\calm_{(p)} \to \calm^1_{(p)}$ is by definition surjective and  Lemma~\ref{lem:lemma2}  implies it is injective.
\end{proof}

\begin{lemma}\label{lem:lemma4}  The are equalities
\[\big|  \calm^1_{(p)}  \big|  =  \big|  \hp^1 \big| = p^{nk/2} =  \big|  \hp^2 \big| =  \big|  \calm^2_{(p)}  \big|.\]
\end{lemma}
\begin{proof}   The product of the orders  $\big| \hp^1 \big|$ and $\big|  \hp^2 \big|$ is $p^{kn}$.  Thus, one of them, say $\hp^1$, is of order at most $p^{nk/2}$.  But $\hp^1$ contains the subgroup $\calm^1_{(p)}$ of order $p^{nk/2}$.  If follows that  $\big| \hp^1 \big| = p^{nk/2}$.  It now follows that we also have that  $\big| \hp^2 \big| = p^{nk/2}$.
\end{proof}

\begin{proof}[Conclusion of Proof of Main Theorem]  Choose a prime $p$ that divides $n$; this implies that $\calh_{p}(P(U)) \cong \Z_{p^k}$ for some $k \ge 1$.  Choose a knot $K$ among the $K_i$ or $L_i$ and a $\chi \ne 0 $ of order a power of $p$   for which $\ot(P(U), \chi) + 2\sigma_{\chi(V_1(P))}(K)$ has the minimum value among all such $\chi$.    This will be either $K_1$ or $L_1$.  Assume it is $K_1$.  The minimum might be achieved for different values of $\chi$, but it is not attained by any   $L_i$.  We denote one such minimizing $\chi$ by $\chi_1$

We have that $\hp   \subset  (\Z_{p^k})^{n} \oplus  (\Z_{p^k})^{n} $.   There exists an  element
\[  (\chi, 0 , \ldots 0) \oplus (\chi_1' , \cdots , \chi_{n/2}') \in \calm_{(p)} \]   for  some set of $\chi_i'$.  If all the $\chi_i'$ are 0, then the Casson-Gordon invariant associated to the element in the  left summand is positive and the invariant for the element in the right summand is 0, so they cannot be equal.  On the other hand, if any $\chi_i'$ is nonzero, then the Casson-Gordon invariant associated the element in the right summand is greater than that for the left summand.
\end{proof}

\bibliographystyle{amsplain}	
\bibliography{../../BibTexComplete}

\end{document}